\documentclass[12pt,twoside]{amsart}
\usepackage{amsmath}
\usepackage{amsthm}
\usepackage{amsfonts}
\usepackage{amssymb}
\usepackage{latexsym}
\usepackage{mathrsfs}
\usepackage{amsmath}
\usepackage{amsthm}
\usepackage{amsfonts}
\usepackage{amssymb}
\usepackage{latexsym}
\usepackage{geometry}
\usepackage{dsfont}
\usepackage[dvips]{graphicx}
\usepackage{color}
\usepackage[all]{xy}

\date{}
\allowdisplaybreaks[4] \footskip=15pt
\renewcommand{\uppercasenonmath}[1]{}

\usepackage{graphicx,amssymb}
\usepackage[all]{xy}
\usepackage{amsmath}

\allowdisplaybreaks
\usepackage{amsthm}
\usepackage{color}

\theoremstyle{plain}
\newtheorem{theorem}{Theorem}[section]
\newtheorem{proposition}[theorem]{Proposition}
\newtheorem{lemma}[theorem]{Lemma}
\newtheorem{corollary}[theorem]{Corollary}
\theoremstyle{definition}
\newtheorem{example}[theorem]{Example}
\newtheorem{definition}[theorem]{Definition}

\theoremstyle{definition}

\theoremstyle{remark}

\newcommand{\pf}{\noindent\begin {proof}}
\newcommand{\epf}{\end{proof}}

\newcommand{\Ker}{\mbox{\rm Ker}}
\newcommand{\Ext}{\mbox{\rm Ext}}
\newcommand{\Hom}{\mbox{\rm Hom}}
\newcommand{\Tor}{\mbox{\rm Tor}}

\newcommand{\C}{\mathcal{C}}

\newcommand{\prodi}{\prod_{i\in I}}

\newcommand{\Id}{\mathrm{Id}}

\def\ra{\rightarrow}

\def\Hom{{\rm Hom}}
\def\Ext{{\rm Ext}}
\def\Tor{{\rm Tor}}

\def\m{{\frak m}}
\def\p{{\frak p}}

\def\Ker{{\rm Ker}}

\def\Im{{\rm Im}}
\def\Coker{{\rm Coker}}

\def\Spec{{\rm Spec}}
\def\Max{{\rm Max}}

\begin{document}
\begin{center}
{\large  \bf On uniformly $S$-coherent rings}

\vspace{0.5cm}  \ Xiaolei Zhang$^{a}$


{\footnotesize 
a.\ \ School of Mathematics and Statistics, Shandong University of Technology, Zibo 255049, China\\
  E-mail: zxlrghj@163.com
}
\end{center}

\bigskip
\centerline { \bf  Abstract}
\bigskip
\leftskip10truemm \rightskip10truemm \noindent

In this paper, we introduce and study  the notions of uniformly $S$-finitely presented modules and uniformly $S$-coherent rings (modules) which are ``uniform'' versions of ($c$-)$S$-finitely presented modules and ($c$-)$S$-coherent rings (modules)  introduced by Bennis and Hajoui \cite{bh18}. Among the results, uniformly $S$-versions of Chase's result, Chase Theorem  and Matlis Theorem  are obtained.
\vbox to 0.3cm{}

\leftskip0truemm \rightskip0truemm
\bigskip
\section{Introduction}
Throughout this paper, all rings are commutative with identity. Let $R$ be a ring.  For a subset $U$ of  an $R$-module $M$, we denote by $\langle U\rangle$ the submodule of $M$ generated by $U$. A subset $S$ of $R$ is called a multiplicative subset of $R$ if $1\in S$ and $s_1s_2\in S$ for any $s_1\in S$, $s_2\in S$.

The study of commutative rings in terms of multiplicative sets began with  Anderson and Dumitrescu \cite{ad02}, who introduced the notion of $S$-Noetherian rings. Recall that a ring $R$ is called an $S$-Noetherian ring if for any ideal $I$ of $R$, there is a finitely generated sub-ideal $K$ of $I$  such that $sI\subseteq K$ for some $s\in S$.   Cohen's Theorem, Eakin-Nagata Theorem and Hilbert Basis Theorem for $S$-Noetherian rings are also given in \cite{ad02}. However, the element $s\in S$ in the definition of $S$-Noetherian rings is not ``uniform'' in general. This situation make it difficult to study $S$-Noetherian rings via module-theoretic methods. To overcome this difficulty, Qi et al. \cite{QKWCZ21} defined uniformly $S$-Noetherian rings as $S$-Noetherian rings in which definition the choice of $s$ is fixed. Then they characterized uniformly $S$-Noetherian rings using $u$-$S$-injective modules.

Recall from \cite{g} that  a ring $R$ is said to be a coherent ring provided that any finitely generated ideal is finitely presented. The notion of coherent rings, which is a generalization of Noetherian rings,  is  another important rings defined by finiteness condition.  Many algebraists studied coherent rings in terms of  various of modules.  Early in 1960, Chase \cite[Theorem 2.1]{C60} showed that a ring is coherent exactly when the class of flat modules is closed under direct product. In 1970 Stenstr\"{o}m \cite[Theorem 3.2]{S70} obtained coherent rings are rings when any direct limits of  absolutely pure modules is  absolutely pure. In 1982, Matlis \cite[Theorem 1]{M82} proved that a ring $R$ is  coherent   if and only if $\Hom_R(M,E)$ is  flat for any injective modules $M$ and $E$.

To extend coherent rings by multiplicative sets, Bennis et al. \cite{bh18} introduced the notions of $S$-coherent rings and $c$-$S$-coherent rings. They also gave an $S$-version of Chase's result to characterize $S$-coherent rings using ideals.  Recently, the authors in this paper et al.\cite{QZZ22} characterized $S$-coherent rings in terms of $S$-Mittag-Leffler modules and  $S$-flat modules (which can be seen as flat modules by localizing at $S$).

The main motivation of this paper is to introduce and study the ``uniform'' version of  $S$-coherent rings for extending uniformly $S$-Noetherian rings. The organization of the paper is as follows: In Section 2, we introduce and study uniformly $S$-finitely presented modules and their connections with $u$-$S$-flat modules and $u$-$S$-projective modules (see Proposition \ref{usfp+usf-usp}). In Section 3, we introduce uniformly $S$-coherent modules and uniformly $S$-coherent rings. In particular, we study the ideal-theoretic characterizations of uniformly $S$-coherent rings (see Proposition \ref{nonn-coh}). Moreover examples of $S$-coherent rings and $c$-$S$-coherent rings which are not uniformly $S$-coherent of are provided (see Example \ref{not-uscoh}). In Section 4, Chase Theorem  and Matlis Theorem for uniformly $S$-coherent rings  are obtained (see Theorem \ref{w-coh-chase} and Theorem \ref{phi-coh-fp}).

Since the paper involves uniformly torsion theory, we give a quick review (see  \cite{zwz21} for more details). An $R$-module $T$ is called  $u$-$S$-torsion (with respect to $s$) provided that there exists  $s\in S$ such that $sT=0$. An $R$-sequence $\cdots\rightarrow A_{n-1}\xrightarrow{f_n} A_{n}\xrightarrow{f_{n+1}} A_{n+1}\rightarrow\cdots$ is $u$-$S$-exact, if for any $n$ there is an element $s\in S$ such that $s\Ker(f_{n+1})\subseteq \Im(f_n)$ and $s\Im(f_n)\subseteq \Ker(f_{n+1})$.
An $R$-sequence $0\rightarrow A\xrightarrow{f} B\xrightarrow{g} C\rightarrow 0$ is called a short $u$-$S$-exact sequence (with respect to $s$), if $s\Ker(g)\subseteq \Im(f)$ and $s\Im(f)\subseteq \Ker(g)$ for some $s\in S$. An $R$-homomorphism $f:M\rightarrow N$ is an \emph{$u$-$S$-monomorphism} $($resp.,   \emph{$u$-$S$-epimorphism}, \emph{$u$-$S$-isomorphism}$)$  (with respect to $s$) provided $0\rightarrow M\xrightarrow{f} N$   $($resp., $M\xrightarrow{f} N\rightarrow 0$, $0\rightarrow M\xrightarrow{f} N\rightarrow 0$ $)$ is  $u$-$S$-exact  (with respect to $s$).
Suppose $M$ and $N$ are $R$-modules. We say $M$ is $u$-$S$-isomorphic to $N$ if there exists a $u$-$S$-isomorphism $f:M\rightarrow N$. A family $\C$  of $R$-modules  is said to be closed under $u$-$S$-isomorphisms if $M$ is $u$-$S$-isomorphic to $N$ and $M$ is in $\C$, then $N$ is  also in  $\C$. One can deduce from the following \cite[Lemma 2.1]{zwz21-p} that the existence of $u$-$S$-isomorphisms of two $R$-modules is actually an equivalence relationship.

\section{uniformly $S$-finitely presented modules}

Recall from \cite{ad02} that an $R$-module $M$ is called  $S$-finite (with respective to $s$) provided that there is an element $s\in S$ and a finitely generated $R$-module $F$ such that $sM\subseteq F\subseteq M$. Trivially,  $S$-finite modules are generalizations of finitely generated modules.
For generalizing finitely presented $R$-modules, Bennis et al. \cite{bh18} introduced the notions of $S$-finitely presented modules and $c$-$S$-finitely presented modules. Following \cite[Definition 2.1]{bh18} that an  $R$-module $M$ is called $S$-finitely presented
provided that there exists an exact sequence of $R$-modules $0\rightarrow K\rightarrow F\rightarrow M\rightarrow 0$ with $K$ $S$-finite and $F$ finitely generated free. Certainly, an  $R$-module $M$  is $S$-finitely presented if and only if   there exists an exact sequence of $R$-modules $0\rightarrow T_1\rightarrow N\rightarrow M\rightarrow 0$ with $N$ finitely presented and $sT_1=0$  for some $s\in S$. Following \cite[Definition 4.1]{bh18} that an  $R$-module $M$ is called $c$-$S$-finitely presented
provided that there exists a finitely presented submodule $N$ of $M$ such that $sM\subseteq N\subseteq M$ for some $s\in S$.  Trivially, an  $R$-module $M$ is called $c$-$S$-finitely presented if and only if  there exists an exact sequence of $R$-modules $0\rightarrow N\rightarrow M\rightarrow T_2\rightarrow 0$ with $N$ finitely presented and $sT_2=0$  for some $s\in S$. Next we will give the notion of uniformly $S$-finitely presented modules which generalize both  $S$-finitely presented modules and  $c$-$S$-finitely presented modules.

\begin{definition}\label{us-FP}
Let $R$ be a ring, $S$ a multiplicative subset of $R$ and $s\in S$. An  $R$-module $M$ is called
$u$-$S$-finitely presented (abbreviates uniformly  $S$-finitely presented) (with respective to $s$) provided that there is an exact sequence $$0\rightarrow T_1\rightarrow F\xrightarrow{f} M\rightarrow T_2\rightarrow 0$$ with $F$ finitely presented and $sT_1=sT_2=0$.
\end{definition}
Trivially,   $S$-finitely presented modules and  $c$-$S$-finitely presented modules are all $u$-$S$-finitely presented.  Certainly, every
$u$-$S$-finitely presented $R$-module is $S$-finite. Indeed, since in definition \ref{us-FP} we have $sT_2=0$, so $sM\subseteq \Im(f)$. Note that the fact that $\Im(f)$ is finitely generated implies  $M$ is $S$-finite.

By \cite[Lemma 2.1]{zwz21-p}, an  $R$-module $M$ is $u$-$S$-finitely presented if and only if   there is an exact sequence $0\rightarrow T_1\rightarrow M\xrightarrow{g} F\rightarrow T_2\rightarrow 0$ with $F$ finitely presented and $s'T_1=s'T_2=0$ for some $s'\in S$. So an $R$-module $M$ is $u$-$S$-finitely presented if and only it is $u$-$S$-isomorphic to a finitely presented $R$-module.

\begin{theorem}\label{prop-usfp} Let $\Phi:0\rightarrow M\xrightarrow{f} N\xrightarrow{g} L\rightarrow 0$ be a $u$-$S$-exact sequence of $R$-modules. The following statements hold.
\begin{enumerate}
\item The class of $u$-$S$-finitely presented modules is closed under $u$-$S$-isomorphisms.
    \item If $M$ and $L$ are $u$-$S$-finitely presented, so is $N$.
       \item Any finite direct sum of $u$-$S$-finitely presented modules  is $u$-$S$-finitely presented.
 \item If $N$ is $u$-$S$-finitely presented, then  $L$ is  $u$-$S$-finitely presented if and only if  $M$ is $S$-finite.
\end{enumerate}
Moreover, if $\Phi$ is an exact sequence, the both side of conditions in  $(2)$ and $(4)$  can be taken to be  ``uniform'' with respective to a same $s\in S$.
\end{theorem}
\begin{proof} (1) It follows from the fact that  an $R$-module $M$ is $u$-$S$-finitely presented if and only it is $u$-$S$-isomorphic to a finitely presented $R$-module.

(2) Since $u$-$S$-finitely presented modules are closed under $u$-$S$-isomorphisms,  we may assume $\Phi$ is an exact sequence by (1). Consider the following push-out:
$$\xymatrix@R=20pt@C=25pt{
  0 \ar[r]^{}&M \ar[d]^{h}\ar[r]^{f}&N \ar[r]^{g}\ar[d]^{l}&L\ar[r] \ar@{=}[d] &0\\
0 \ar[r]^{}&F_1 \ar[r]^{m}&X \ar[r]^{n}&L \ar[r] &0.\\}$$
with $F_2$ finitely presented, $\Ker(h)$ and $\Coker(h)$ $u$-$S$-torsion. So $l$ is also a $u$-$S$-isomorphism.
Consider the following pull-back:
$$\xymatrix@R=20pt@C=25pt{
  0 \ar[r]^{}&F_1 \ar[r]^{m}&X \ar[r]^{n}&L\ar[r]  &0\\
0 \ar[r]^{}&F_1 \ar@{=}[u]\ar[r]&Y \ar[u]^{k}\ar[r]&F_2 \ar[r]\ar[u]^{j} &0.\\}$$
with $F_2$ finitely presented, $\Ker(j)$ and $\Coker(j)$ $u$-$S$-torsion. So $k$ is also a $u$-$S$-isomorphism. Since $F_1$ and $F_2$ are finitely presented, $Y$ is also finitely presented. Hence $N$ is $u$-$S$-isomorphic to a  finitely presented $R$-module, and thus is $u$-$S$-finitely presented.

$(3)$ Follows from (2).

$(4)$ Since $u$-$S$-finitely presented modules and $S$-finite modules are closed under $u$-$S$-isomorphisms respectively, we may assume $\Phi$ is an exact sequence by (1). Suppose $M$ is $S$-finite. Since $N$ is $u$-$S$-finitely presented,  there is an exact sequence $0\rightarrow T_1\rightarrow F\xrightarrow{l} N\rightarrow T_2\rightarrow 0$ with  $F$ finitely presented and $sT_1=sT_2=0$ for some $s\in S$.
Consider the following pull-back of $f$ an $l$:
$$\xymatrix@R=20pt@C=25pt{
  0 \ar[r]^{}&M \ar[r]^{f}&N \ar[r]^{g}&L\ar[r]  &0\\
0 \ar[r]^{}&Z\ar[r]\ar[u]^{s}&F \ar[u]^{l}\ar[r]&K\ar[u]^{t} \ar[r] &0.\\}$$
Since $l$ is a $u$-$S$-isomorphism, $s$ and $t$ are both  $u$-$S$-isomorphisms. So $Z$ is also  $S$-finite.
Note that $L$ is   $u$-$S$-isomorphic to $K$ which is $u$-$S$-finitely presented (see \cite[Theorem 2.4(4)]{bh18}). So $L$ is  $u$-$S$-finitely presented. Suppose $L$ is  $u$-$S$-finitely presented. Considering the above pull-back, we have  $K$ is also $S$-finitely presented. Hence $Z$ is  $S$-finite by  \cite[Theorem 2.4(5)]{bh18}  which implies that $M$ is also $S$-finite.

The ``Moreover'' part can be checked by the proof of $(2)$ and $(4)$.
\end{proof}

Recall from \cite{brt18} that an $R$-module $M$ is said to be $S$-Noetherian provided that any submodule of $M$ is $S$-finite. A ring $R$ is called $S$-Noetherian if $R$ itself is  $S$-Noetherian  $R$-module.
\begin{proposition}\label{snoe}Let $R$ be a ring and $S$ a multiplicative subset of $R$. Then a ring $R$ is $S$-Noetherian if and only if any $S$-finite module is $u$-$S$-finitely presented.
\end{proposition}
\begin{proof} For necessity, let $M$ be an $S$-finite module. Then there is a $u$-$S$-epimorphism $f:F\rightarrow M$ with $F$ finitely generated free. Since $R$ is an $S$-Noetherian ring, we have $F$ is also $S$-Noetherian (see \cite{brt18}). Hence $M$ is $u$-$S$-finitely presented by Theorem \ref{prop-usfp}(4). For sufficiency, let $I$ be an ideal of $R$. Then $R/I$ is $S$-finite, and thus $u$-$S$-finitely presented. By  Theorem \ref{prop-usfp}(4) again, $I$ is $S$-finite.
\end{proof}

\begin{proposition}\label{sfp-csfp-usfp} Let $R$ be a ring, $S$ a multiplicative subset of $R$ consisting of finite elements. Then an $R$-module $M$ is a $u$-$S$-finitely presented $R$-module if and only if $M_S$ is a finitely  presented $R_S$-module.
\end{proposition}
\begin{proof} Suppose $M$ is a $u$-$S$-finitely presented $R$-module. there is an exact sequence $0\rightarrow T_1\rightarrow N\xrightarrow{f} M\rightarrow T_2\rightarrow 0$ with $N$ finitely presented and $sT_1=sT_2=0$. Localizing at $S$, we have $0\rightarrow (T_1)_S\rightarrow N_S\xrightarrow{f} M_S\rightarrow (T_2)_S\rightarrow 0$. Since $sT_1=sT_2=0$, $(T_1)_S=(T_2)_S=0$. So $M_S\cong N_S$ is  a finitely generated $R_S$-module. On the other hand, suppose $M_S$ is a finitely generated $R_S$-module. Suppose $S=\{s_1,\cdots,s_n\}$ and set $s=s_1\cdots s_n$. We may assume that $M_S$ is generated by $\{\frac{m_1}{s},\cdots,\frac{m_n}{s}\}$. Consider the $R$-homomorphism $f:R^n\rightarrow M$ satisfying $f(e_i)=m_i$ for each $i=1,\cdots,n$. It is easy to verify $f$ is a $u$-$S$-epimorphism. Consider the exact sequence $0\rightarrow \Ker(f_S)\rightarrow R^n_S\xrightarrow{f_S} M_S\rightarrow 0$. Then $\Ker(f_S)$ is a finitely generated $R_S$-module, and thus $\Ker(f)$ is $S$-finite. By Theorem \ref{prop-usfp}(2), $M$ is $u$-$S$-finitely presented.
\end{proof}

Let $\p$ be a prime ideal of $R$. We say an $R$-module $M$  is (simply)  \emph{$\p$-finite} provided  $R$ is $(R\setminus\p)$-finite. We always denote by $\Spec(R)$ the spectrum of all prime ideals of $R$, and $\Max(R)$ the set of all maximal ideals of $R$, respectively.

\begin{lemma}\label{char-fg} Let $R$ be a ring, $S$ a multiplicative subset of $R$ and $M$ an $R$-module. The following statements are equivalent:
\begin{enumerate}
    \item $M$ is finitely generated $R$-module;
     \item $M$ is  $\p$-finite for any $\p\in\Spec(R)$;
      \item $M$ is  $\m$-finite for any $\m\in\Max(R)$.
\end{enumerate}
\end{lemma}
\begin{proof} $(1)\Rightarrow (2)\Rightarrow (3)$ Trivial.

$(3)\Rightarrow (1)$  For each  $\m\in \Max(R)$, there exists an element $s^{\m}\in R\setminus\m$ and a finitely generated submodule $F^{\m}$ of $M$ such that $s^{\m}M\subseteq F^{\m}$. Since $\{s^{\m} \mid \m \in \Max(R)\}$ generated $R$, there exist finite elements $\{s^{\m_1},...,s^{\m_n}\}$ such that $\langle s^{\m_1},...,s^{\m_n}\rangle=R$. So $M=\langle s^{\m_1},...,s^{\m_n}\rangle M\subseteq F^{\m_1}+...+F^{\m_n}\subseteq M$. Hence $M=F^{\m_1}+...+F^{\m_n}$. It follows that $M$ is finitely generated.
\end{proof}

Let $\p$ be a prime ideal of $R$. We say an $R$-module $M$  is (simply)  \emph{$u$-$\p$-finitely presented} provided  $R$ is $u$-$(R\setminus\p)$-finitely presented.
\begin{proposition}\label{char-fp} Let $R$ be a ring, $S$ a multiplicative subset of $R$ and $M$ an $R$-module. The following statements are equivalent:
\begin{enumerate}
    \item $M$ is a finitely presented $R$-module;
     \item $M$ is  $u$-$\p$-finitely presented for any $\p\in\Spec(R)$;
      \item $M$ is  $u$-$\m$-finitely presented for any $\m\in\Max(R)$.
\end{enumerate}
\end{proposition}
\begin{proof} $(1)\Rightarrow (2)\Rightarrow (3)$ Trivial.

$(3)\Rightarrow (1)$  By Lemma \ref{char-fg}, $M$ is finitely generated. Consider the exact sequence $0\rightarrow K\rightarrow F\rightarrow M\rightarrow 0$ with $F$ finitely generated free. By Theorem \ref{prop-usfp}, $K$ is $\m$-finite for any $\m\in\Max(R)$. So $K$ is also finitely generated, and thus $M$ is finitely presented.
\end{proof}

Let $\{M_j\}_{j\in \Gamma}$  be a family of $R$-modules and $N_j$ a submodule of $M_j$ generated by $\{m_{i,j}\}_{i\in \Lambda_j}\subseteq M_j$ for each $j\in \Gamma$. Recall from \cite{zwz21} that  a family of $R$-modules  $\{M_j\}_{j\in \Gamma}$  is \emph{$u$-$S$-generated} (with respective to $s$) by  $\{\{m_{i,j}\}_{i\in \Lambda_j}\}_{j\in \Gamma}$ provided that there exists an element $s\in S$  such that $sM_j\subseteq N_j$ for each $j\in \Gamma$, where $N_j=\langle \{m_{i,j}\}_{i\in \Lambda_j}\rangle$.  We say a family of $R$-modules  $\{M_j\}_{j\in \Gamma}$ is \emph{$u$-$S$-finite} (with respective to $s$) if the set $\{m_{i,j}\}_{i\in \Lambda_j}$ can be chosen as a finite set for each $j\in \Gamma$, that is, there is $s\in S$ such that  $\{M_j\}_{j\in \Gamma}$ are all $S$-finite with respect to $s$. Recall from \cite{QKWCZ21} that an $R$-module $M$ is called a \emph{$u$-$S$-Noetherian module} provided the set of all submodules of $M$ is $u$-$S$-finite. A ring $R$ is called to be a \emph{$u$-$S$-Noetherian ring} provided that $R$ itself is a $u$-$S$-Noetherian $R$-module.

\begin{theorem}\label{usnoe}Let $R$ be a ring and $S$ a multiplicative subset of $R$. Then the following statements are equivalent:
\begin{enumerate}
    \item A ring $R$ is $u$-$S$-Noetherian;
   \item Any $S$-finite module is $u$-$S$-Noetherian;
    \item Any finitely generated module is $u$-$S$-Noetherian;
   \item There is $s\in S$ such that any finitely generated module is $u$-$S$-finitely presented with respective to $s$.
 \end{enumerate}
\end{theorem}
\begin{proof} $(1)\Rightarrow (2)$ Let $M$ be an $S$-finite module. Then there is a $u$-$S$-epimorphism $f:F\rightarrow M$ with $F$ finitely generated free. Since $R$ is $u$-$S$-Noetherian, we have $F$ is also $u$-$S$-Noetherian, and so is $M$ (see \cite[Proposition 2.13]{QKWCZ21}).

 $(2)\Rightarrow (3)\Rightarrow (4)$  Trivial.

 $(4)\Rightarrow (1)$ Let $I$ be an ideal of $R$. Then $R/I$ is $u$-$S$-finitely presented with respective to $s$. So $I$ is $S$-finite with respective to $s$ by Theorem \ref{prop-usfp}(4), implying $R$ is $u$-$S$-Noetherian.
\end{proof}

Recall from \cite{zwz21,zwz21-p} that an $R$-module $P$ is called \emph{$u$-$S$-projective} (resp., \emph{ $u$-$S$-flat)} provided that the induced sequence $0\rightarrow \Hom_R(P,A)\rightarrow \Hom_R(P,B)\rightarrow \Hom_R(P,C)\rightarrow 0$ (resp., $0\rightarrow  P\otimes_RA\rightarrow  P\otimes_RB\rightarrow  P\otimes_RC\rightarrow  0$) is $u$-$S$-exact for any $u$-$S$-exact sequence $0\rightarrow A\rightarrow B\rightarrow C\rightarrow 0$. It was proved in \cite[Proposition 2.9]{zwz21-p} that any  $u$-$S$-projective module is $u$-$S$-flat.
\begin{proposition}\label{usfp+usf-usp}Let $R$ be a ring and $S$ a multiplicative subset of $R$. Then the following statements hold.
\begin{enumerate}
   \item Every $S$-finite  $u$-$S$-projective module is $u$-$S$-finitely presented.
   \item Every $u$-$S$-finitely presented $u$-$S$-flat module is  $u$-$S$-projective.
 \end{enumerate}
\end{proposition}
\begin{proof} (1) Let $P$ be an $S$-finite  $u$-$S$-projective module, then there is a $u$-$S$-exact sequence $\Psi: 0\rightarrow \Ker(f)\xrightarrow{i} F\xrightarrow{f} P\rightarrow 0$ with $F$ finitely generated free. Since $P$ is $u$-$S$-projective, the sequence $\Psi$ is $u$-$S$-split by \cite[Theorem 2.7]{zwz21-p}. So there is a $u$-$S$-epimorphism $i':F\rightarrow \Ker(f)$ such that $i'\circ i=s\Id_{\Ker(f)}$ for some $s\in S$. Hence $\Ker(f)$ is $S$-finite, and so $P$ is $u$-$S$-finitely presented by Theorem \ref{prop-usfp}.

(2) Let $M$ be a  $u$-$S$-finitely presented $u$-$S$-flat module. Then there is a  $u$-$S$-exact sequence $\Upsilon: 0\rightarrow \Ker(f)\xrightarrow{i} F\xrightarrow{f} M\rightarrow 0$ with  $F$ finitely generated free and $\Ker(f)$ $S$-finite. Since $M$ is $u$-$S$-flat, $\Upsilon$ is $u$-$S$-pure by \cite[Proposition 2.4]{zousap}. It follows from \cite[Theorem 2.2]{zousap} that $\Upsilon$ is $u$-$S$-split. Thus $M$ is $u$-$S$-projective.
\end{proof}

\section{uniformly $S$-coherent modules and uniformly $S$-coherent rings}

Recall that an $R$-module is said to be a \emph{coherent module} if it is finitely generated and any finitely generated submodule is finitely presented. A ring $R$ is said to be a \emph{coherent ring} if $R$ is a coherent $R$-module. In this section, we will introduce a ``uniform'' version of coherent rings and coherent modules.

\begin{definition}\label{def-us-cohm}
Let $R$ be a ring and $S$ a multiplicative subset of $R$ . An  $R$-module $M$ is called a
\emph{$u$-$S$-coherent module} (abbreviates uniformly  $S$-coherent) (with respective to $s$) provided that there is $s\in S$ such that it is $S$-finite with respect to $s$ and  any finitely generated submodule of $M$ is $u$-$S$-finitely presented with respective to $s$.
\end{definition}

\begin{theorem}\label{prop-uscohm} Let $\Phi:0\rightarrow M\xrightarrow{f} N\xrightarrow{g} L\rightarrow 0$ be a $u$-$S$-exact sequence of $R$-modules. The following statements hold.
\begin{enumerate}
\item The class of $u$-$S$-coherent modules is closed under $u$-$S$-isomorphisms.
 \item If $L$ is $u$-$S$-coherent, then  $M$ is  $u$-$S$-coherent if and only if  $N$ is  $u$-$S$-coherent.
 \item Any finite direct sum of $u$-$S$-coherent modules  is $u$-$S$-coherent.
\item If $N$ is $u$-$S$-coherent and $M$ is $S$-finite, then $L$ is $u$-$S$-coherent.
\end{enumerate}
\end{theorem}
\begin{proof} (1) Let $h:A\rightarrow B$ be a $u$-$S$-isomorphism with $s_1\Ker(h)=s_1\Coker(h)=0$. Suppose $B$ is  $u$-$S$-coherent with respective to $s_2$, then one can check $A$ is  $u$-$S$-coherent with respective to $s_1s_2$. Similarly, if  $A$ is  $u$-$S$-coherent, then $B$ is also $u$-$S$-coherent (see \cite[Lemma 2.1]{zwz21-p}).

(2) By (1), we can assume that $\Phi$ is an exact sequence. Suppose $M$ and $L$ are  $u$-$S$-coherent with respective to $s$. Then  one can check $N$ is $u$-$S$-coherent with respective to $s$ from the proof of Theorem \ref{prop-usfp}(2). Suppose $N$ and $L$ are  $u$-$S$-coherent with respective to $s$. Then $M$ is $S$-finite with respective to some $s\in S$ by Theorem \ref{prop-usfp}(4). Since $N$ is $u$-$S$-coherent with respective to $s$,   $M$ is $u$-$S$-coherent with respective to $s$.

(3) Follows by (2).

(4) Assume that $\Phi$ is an exact sequence. Suppose $N$ is $u$-$S$-coherent with respective to $s$ and $M$ is $S$-finite  with respective to $s$ for some $s\in S$. Then $L$ is also $S$-finite  with respective to $s$. Let $K$ be a finitely generated submodule of $L$. Then the sequence   $0\rightarrow M \rightarrow g^{-1}(K) \rightarrow K\rightarrow 0$ is exact. So $g^{-1}(K)$ is  $S$-finite. Consider the following commutative diagram with rows and columns exact:
$$\xymatrix@R=17pt@C=25pt{ &0\ar[d] &0\ar[d]&0\ar[d]&\\
 0\ar[r]^{}& \Ker(m) \ar[d]^{}\ar[r]^{}&\Ker(n) \ar[d]^{}\ar[r]^{}& K_1\ar[d]^{}\ar[r]&0  \\
  0 \ar[r]^{}& R^n\ar[d]^{m}\ar[r]^{}& R^{n+s} \ar[r]^{}\ar[d]^{n}&R^{s}\ar[r] \ar[d]^{} &0\\
0\ar[r]^{}&M\ar[r]^{}&g^{-1}(K) \ar[r]&K \ar[r]\ar[d] &0\\
 &  & &0 &\\}$$
where $m$ and $n$ are $u$-$S$-epimorphisms. Since  $N$ is $u$-$S$-coherent, $\Ker(n)$ is $S$-finite, and so is $K_1$. Thus $L$ is $u$-$S$-coherent (with respective to $s$).
\end{proof}

\begin{corollary}\label{abel-uscohm} Let  $f:M\rightarrow N$ be an $R$-homomorphism of  $u$-$S$-coherent modules $M$ and $N$. Then $\Ker(f)$, $\Im(f)$ and $\Coker(f)$ are also $u$-$S$-coherent.
 \end{corollary}
\begin{proof}  Using Theorem \ref{prop-uscohm} and the exact sequences $0\rightarrow \Ker(f)\rightarrow M\rightarrow \Im(f)\rightarrow 0$ and  $0\rightarrow \Im(f)\rightarrow N\rightarrow \Coker(f)\rightarrow 0$.
\end{proof}

\begin{corollary}\label{cap-add} Let $M$ and $N$ be $u$-$S$-coherent sub-modules of a $u$-$S$-coherent module. Then $M+N$ is  $u$-$S$-coherent if and only if so is $M\cap N$.
 \end{corollary}
\begin{proof} Following by Theorem \ref{prop-uscohm} and the exact sequence $0\rightarrow M\cap N\rightarrow M\oplus N\rightarrow M+ N\rightarrow 0$.
\end{proof}

Let $\p$ be a prime ideal of $R$. We say an $R$-module $M$  is (simply)  \emph{$u$-$\p$-coherent} provided  $R$ is $u$-$(R\setminus\p)$-coherent.
\begin{proposition}\label{char-cohm} Let $R$ be a ring, $S$ a multiplicative subset of $R$ and $M$ an $R$-module. The following statements are equivalent.
\begin{enumerate}
    \item $M$ is a coherent $R$-module.
     \item $M$ is  $u$-$\p$-coherent for any $\p\in\Spec(R)$.
      \item $M$ is  $u$-$\m$-coherent for any $\m\in\Max(R)$.
\end{enumerate}
\end{proposition}
\begin{proof} $(1)\Rightarrow (2)\Rightarrow (3)$ Trivial.

$(3)\Rightarrow (1)$ By   Lemma \ref{char-fg}, $M$ is finitely generated. Let $N$ be a finitely generated of $M$, then $M$ is $u$-$\m$-finitely presented for any $\m\in\Max(R)$. So $M$ is finitely presented by Proposition \ref{char-fp}.
\end{proof}

\begin{definition}\label{def-us-cohr}
Let $R$ be a ring, $S$ a multiplicative subset of $R$ and $s\in S$. Then  $R$ is called a
\emph{$u$-$S$-coherent ring} (abbreviates uniformly  $S$-coherent) ring (with respective to $s$) provided that $R$ itself is a uniformly  $S$-coherent $R$-module with respective to $s$.
\end{definition}

Trivially, every coherent ring is $u$-$S$-coherent for any multiplicative set $S$. And if $S$ is composed of units, then $u$-$S$-coherent rings are exactly coherent rings.

The proof of the following result is easy and direct, so we omit it.
\begin{lemma}\label{abel-uscohm} Let $R=R_1\times R_2$ be direct  product of rings $R_1$ and $R_2$, $S=S_1\times S_2$ a multiplicative subset of $R$. Then $R$ is $u$-$S$-coherent if and only if $R_i$ is $u$-$S_i$-coherent for any $i=1,2$.
 \end{lemma}

The following example shows that not every $u$-$S$-coherent rings is coherent.

\begin{example}\label{not-coh} Let $R_1$ be a coherent ring and $R_2$ a non-coherent ring, $S_1=\{1\}$ and $S_2=\{0\}$. Set $R=R_1\times R_2$ and $S=S_1\times S_2$. Then $R$ is $u$-$S$-coherent non-coherent ring.
 \end{example}

Let $\p$ be a prime ideal of $R$. We say a ring $R$ is (simply)  \emph{$u$-$\p$-coherent} provided  $R$ is $u$-$(R\setminus\p)$-coherent.
\begin{proposition}\label{char-cohr} Let $R$ be a ring and $S$ a multiplicative subset of $R$ . The following statements are equivalent.
\begin{enumerate}
    \item $R$ is a coherent ring.
     \item $R$ is a $u$-$\p$-coherent ring for any $\p\in\Spec(R)$.
      \item $R$ is a $u$-$\m$-coherent ring for any $\m\in\Max(R)$.
\end{enumerate}
\end{proposition}
\begin{proof} Follows by  Proposition \ref{char-cohm}.
\end{proof}

\begin{proposition}\label{abel-uscohm} Let $R$ be a ring, $S$ a multiplicative subset of $R$. If $R$ is a $u$-$S$-Noetherian ring, then $R$ is $u$-$S$-coherent.
 \end{proposition}
\begin{proof} Follows from Theorem \ref{usnoe}.
\end{proof}

Trivially, $u$-$S$-coherent rings are not $u$-$S$-Noetherian in general. Indeed, we can find a non-Noetherian coherent ring in the case that $S=\{1\}$.

In 1960, Chase characterized coherent rings by considering annihilator of elements and  intersection of  finitely generated  ideals in \cite[Theorem 2.2]{C60}. Now, we give  a ``uniform'' version of Chase's result.

\begin{proposition}\label{nonn-coh}\textbf{$($Chase's result for $u$-$S$-coherent rings$)$}
Let $R$ be a ring and $S$ a multiplicative subset of $R$.  Then   the following assertions are equivalent:
\begin{enumerate}
    \item  $R$ is a $u$-$S$-coherent ring;
      \item there is $s\in S$ such that $(0:_Rr)$ is $S$-finite with respective to $s$ for any  $r\in R$, and the intersection of two finitely generated  ideals of $R$ is $S$-finite with respective to $s$;
      \item  there is $s\in S$ such that $(I:_Rb)$ is $S$-finite with respective to $s$ for any  element $b\in R$ and any   finitely generated  ideal $I$ of $R$.
 \end{enumerate}
\end{proposition}

\begin{proof}  $(1)\Rightarrow (2)$: Suppose  $R$ is $u$-$S$-coherent with respective to $s$. Considering  the exact sequence $0\rightarrow (0:_Rr)\rightarrow R\rightarrow Rr\rightarrow 0$, we have $(0:_Rr)$ is $S$-finite with respective to $s$ by Theorem \ref{prop-usfp}. For any  two finitely generated  ideals $I,J$ of $R$, we have $I\cap J$ is  $S$-finite with respective to $s$ by Corollary \ref{cap-add} and Theorem \ref{prop-usfp}.

$(2)\Rightarrow (1)$: Let $I=\langle a_1,\cdots,a_n\rangle$ be a finitely generated ideal of $R$. We claim that  $I$ is $u$-$S$-finitely presented with respective to $s$ by induction on $n$. Suppose $n=1$. the claim follows by  the exact sequence $0\rightarrow (0:_Rr)\rightarrow R\rightarrow Rr\rightarrow 0$. Suppose $n=k$, the claim holds. Suppose $n=k+1$. the claim holds by the exact sequence $0\rightarrow \langle a_1,\cdots,a_k\rangle\cap \langle a_{k+1}\rangle\rightarrow \langle a_1,\cdots,a_k\rangle\oplus \langle a_{k+1}\rangle \rightarrow \langle a_1,\cdots,a_{k+1}\rangle\rightarrow 0$. So the claim holds for all $n$.

$(1)\Rightarrow (3)$: Suppose  $R$ is $u$-$S$-coherent with respective to $s$. Let $I$ be a finitely generated  ideal of $R$ and $b$ an element in $R$. Consider the following commutative diagram with  exact rows: $$\xymatrix@R=20pt@C=15pt{
 0 \ar[r]^{}  & I\ar[d]_{}\ar[r]^{} &Rb+I\ar[r]^{}\ar[d]_{} & (Rb+I)/I \ar[d]^{\cong}\ar[r]^{} &  0\\
   0 \ar[r]^{} &(I:_Rb) \ar[r]^{} &R \ar[r]^{} & R/(I:_Rb)\ar[r]^{} &  0.}$$
 Since $R$ is $u$-$S$-coherent with respective to $s$, we have $Rb+I$ is $u$-$S$-finitely presented  with respective to $s$. Since $I$ is finitely generated, $(Rb+I)/I$ is $u$-$S$-finitely presented  with respective to $s$  by  Theorem \ref{prop-usfp}. Thus $(I:_R b)$ is $S$-finite is  with respective to $s$  by  Theorem \ref{prop-usfp} again.

$(3)\Rightarrow (1)$: Let  $I$ be a finitely generated ideal of $R$ generated by $\{a_1,...,a_{n}\}$ . We will show $I$ is $u$-$S$-finitely presented by induction on $n$. The case $n=1$ follows from the exact sequence $0\rightarrow (0:_Ra_1)\rightarrow R\rightarrow Ra_1\rightarrow 0$. For $n\geq 2$, let $L=\langle a_1,...,a_{n-1}\rangle$. Consider the exact sequence $0 \ra (L:_R a_n)\ra R\ra (Ra_n+L)/L\ra 0$. Then $(Ra_n+L)/L=I/L$ is $u$-$S$-finitely presented  with respective to $s$  by (3) and Theorem \ref{prop-usfp}. Consider the exact sequence $0 \rightarrow L \rightarrow I \rightarrow I/L\rightarrow 0$. Since $L$ is finitely  presented by induction and $I/L$ is $u$-$S$-finitely presented  with respective to $s$ , $I$ is also $u$-$S$-finitely presented  with respective to $s$ by Theorem \ref{prop-usfp}.
\end{proof}

Recall from \cite{bh18} that a ring $R$ is \emph{$S$-coherent} (resp., \emph{$c$-$S$-coherent}) provided that any finitely generated ideal is $S$-finitely presented (resp., $c$-$S$-finitely presented).
\begin{proposition}\label{abel-uscohm} Let $R$ be a ring, $S$ a multiplicative subset of $R$. If $R$ is a $u$-$S$-coherent ring, then $R$ is both $S$-coherent and $c$-$S$-coherent.
 \end{proposition}
\begin{proof}  Let $I$ be a finitely generated ideal and $0\rightarrow K\rightarrow F\rightarrow I\rightarrow 0$ an exact sequence with $F$ finitely generated free. Then $K$ is $S$-finite by Theorem \ref{prop-usfp}(4). Thus $I$ is $S$-finitely presented, and so $R$ is $S$-coherent. Consider the exact sequence $0\rightarrow T_1\rightarrow N\xrightarrow{f} I\rightarrow T_2\rightarrow 0$ with $N$ finitely presented and $sT_1=sT_2=0$. Note that sin $sT_2=0$, we have $sI\subseteq \Im(f)\cong N/T_1$. Since $sT_1=0$, $s^2I$ can be seen as a submodule of $N$. Hence $I$ is $c$-$S$-finitely presented. Consequently, $R$  is $c$-$S$-coherent.
\end{proof}

\begin{proposition} \label{s-loc-u-noe-fini}
Let $R$ be a ring and $S$ a multiplicative subset of $R$ consisting of finite  elements. Then the following statements are equivalent.
\begin{enumerate}
\item  $R$ is a $u$-$S$-coherent ring.
 \item $R$ is an $S$-coherent ring.
 \item $R$ is a  $c$-$S$-coherent ring.
 \end{enumerate}
\end{proposition}
\begin{proof} Suppose $S=\{s_1,...,s_n\}$ and set $s=s_1\cdots s_n$.

$(1)\Rightarrow (2)$ and  $(1)\Rightarrow (3)$ Follows by Proposition \ref{abel-uscohm}.

$(2)\Rightarrow (1)$  Let $I$ be a finitely generated ideal of $R$. Then there is an exact sequence $0\rightarrow K\rightarrow F\rightarrow I\rightarrow 0$ with $F$ finitely generated free and $K$ $S$-finite.
Let $X$ be a submodule of $K$ such that $s_iK\subseteq X$ for some $s_i\in S$. So $sK/X=0$ Then the exact sequence  $0\rightarrow K/X \rightarrow F/X\rightarrow I\rightarrow 0$ makes $I$ $u$-$S$-finitely presented with respective to $s$. So $R$ is a $u$-$S$-coherent ring.

$(3)\Rightarrow (1)$  Let $I$ be a finitely generated ideal of $R$.  Then there is a finitely presented sub-ideal $J$ of $R$ such that $s_iI\subseteq J=0$. So $s(I/J)=0$. Then the exact sequence  $0\rightarrow I \rightarrow J\rightarrow I/J\rightarrow 0$ makes $I$ $u$-$S$-finitely presented with respective to $s$. So $R$ is a $u$-$S$-coherent ring.
\end{proof}

Let $R$ be a ring, $M$ an $R$-module and $S$ a multiplicative subset of $R$. For any $s\in S$, there is a  multiplicative subset $S_s=\{1,s,s^2,....\}$ of $S$. We denote by $M_s$ the localization of $M$ at $S_s$. Certainly, $M_s\cong M\otimes_RR_s$

\begin{proposition} \label{s-loc-u-noe}
 Let $R$ be a ring and $S$ a multiplicative subset of $R$. If $R$ is a $u$-$S$-coherent ring with respective to some $s\in S$, then  $R_{s}$ is a coherent ring.
\end{proposition}
\begin{proof} Suppose  $R$ is $u$-$S$-coherent ring with respective to  $s\in S$. Let $J$ be a finitely generated ideal of $R_{s}$. Then $J\cong I_s$ for some finitely generated ideal $I$ of $R$. So there is an exact sequence $0\rightarrow T_1\rightarrow K\rightarrow I\rightarrow T_2\rightarrow 0$ with  $K$ finitely presented  and $sT_1=sT_2=0$. Localizing at $S_s$, we have $(T_1)_{s}=(T_2)_{s}=0$. So $J\cong I_s\cong K_s$ which is finitely presented over $R_s$. So $R_{s}$ is a coherent ring.
\end{proof}

Next, we will give an example of a ring which is both $S$-coherent and $c$-$S$-coherent, but not $u$-$S$-coherent.
\begin{example}\label{not-uscoh} Let $R$ be a domain.  Set $S=R-\{0\}$. First, we will show $R$ is  $c$-$S$-coherent. Let $I$ be a nonzero finitely generated ideal of $R$. Suppose $0\not= r\in I$. Then we have  $rI\subseteq Rr\subseteq I$. Since $Rr\cong R$ is finitely presented, $R$ is a $c$-$S$-coherent ring.

Next we will show  $R$ is $S$-coherent.  Let $I$ be a nonzero finitely generated ideal of $R$ generated by nonzero elements $\{a_1,\cdots,a_n\}$. Set $a=a_1\cdots a_n$. Consider the natural exact sequence $0\rightarrow K\rightarrow R^n\xrightarrow{f} I\rightarrow 0$ satisfying $f(e_i)=a_i$ for each $i$. Claim $K$ is $S$-finite with respective to $a$  by induction on $n$. Set $I_k=\langle a_1,\cdots,a_k\rangle$. Suppose $n=1$. Then $K=0$ as $a_1$ is a non-zero-divisor. So the claim trivially holds. Suppose  the claim holds for $n=k$. Now let $n=k+1$. Consider the following commutative diagram with exact rows and columns.
$$\xymatrix@R=20pt@C=25pt{
0 \ar[r]^{} & K_k \ar@{^{(}->}[d]\ar[r]^{} &R^{k} \ar@{^{(}->}[d]\ar[r]^{} &I_k  \ar@{^{(}->}[d]\ar[r]^{} &  0\\
0 \ar[r]^{} & K_{k+1}\ar@{->>}[d]\ar[r]^{} & R^{k+1} \ar@{->>}[d]\ar[r]^{} &I_{k}+Ra_{k+1}\ar@{->>}[d]\ar[r]^{} &  0\\
0 \ar[r]^{} & (I_k:_RRa_{k+1}) \ar[r]^{} & R\ar[r]^{} & (I_{k}+Ra_{k+1})/I_k\ar[r]^{} &  0,\\}$$
Since $a(I_k:_RRa_{k+1})\subseteq aR\subseteq (I_k:_RRa_{k+1})$, So $(I_k:_RRa_{k+1})$ is $S$-finite with respective to $a$. By induction, $K_k$ is $S$-finite with respective to $a$. It is easy to check $K_{k+1}$ is  also $S$-finite with respective to $a$. So the claim holds. Consequently,  $R$ is $S$-coherent.

Now, let  $R$ is a domain such that  $R_{s}$ is not coherent for any $s\not=0$. For example, $R=\mathbb{Q}+x\mathbb{R}[[x]]$ be the subring of formal power series ring $T=\mathbb{R}[[x]]$ with constants in real numbers $\mathbb{R}$, where $\mathbb{Q}$ is the set of all rational numbers.  Indeed, let $0\not=s=a+xf(x)\in R$. We divide it into two cases. Case I: $a\not=0$. In this case, $s$ is a unit in $R$, and so  $R_{s}\cong R$ which is not coherent by \cite[Theorem 5.2.3]{g} . Case II:  $a=0$. In this case, $R_s\cong \mathbb{Q}+(x\mathbb{R}[[x]])_{xf(x)}\cong \mathbb{Q}+(x\mathbb{R}[[x]])_{x}$. So $R_s$ can fit into a Milnor square of type II:
$$\xymatrix@R=32pt@C=35pt{
R_s\ar@{^{(}->}[r]^{}\ar@{->>}[d]_{} & \mathbb{R}[[x]][x^{-1}] \ar@{->>}[d]\\
\mathbb{Q} \ar@{^{(}->}[r]^{} & \mathbb{R}.}$$
Hence $R_s$ is not a coherent domain by \cite[Theorem 8.5.17]{fk16}. 
 We will show $R$ is not a $u$-$S$-coherent ring. On contrary, suppose  $R$ is $u$-$S$-coherent. Then there is a $s\not=0$ such that $R_s$ is a coherent ring by Proposition \ref{s-loc-u-noe}, which is a contradiction.
 \end{example}

\section{module-theoretic characterizations of  uniformly $S$-coherent rings}

In this section, we will characterize uniformly $S$-coherent rings in terms of $u$-$S$-flat modules and $u$-$S$-injective modules.  The following lemma is basic and of independent interest.

\begin{lemma}\label{ele-dire}
Let $R$ be a ring, $r\in R$ and $M$ an $R$-module. Suppose $N$ is a pure submodule of $M$. Then we have the following natural isomorphism  $$\frac{rM}{rN}\cong r(\frac{M}{N}).$$ Consequently, suppose $\{M_i\mid i\in\Lambda\}$ is a direct system of $R$-modules. Then $$r\lim\limits_{\longrightarrow}M_i\cong \lim\limits_{\longrightarrow}(rM_i).$$
\end{lemma}
\begin{proof} Consider the surjective map $f:\frac{rM}{rN}\rightarrow  r(\frac{M}{N})$ defined by $f(rm+rN)=r(m+N)$. It is certainly $R$-linear. We will check it is also well defined. Indeed, $f(rn+rN)=r(n+N)=r(0+N)=0$. So $f$ is an $R$-epimorphism. Let $rm+rN\in \Ker(f)$. Then $rm\in N$. Since  $N$ is a pure submodule of $M$, there is $n\in N$ such that $rm=rn$. So $rm+rN=rn+rN=0$. Hence $f$ is an isomorphism. Suppose $\{(M_i,f_{ij})\mid i,j\in\Lambda\}$ is a direct system of $R$-modules. Then there is a pure exact sequence $0\rightarrow K\rightarrow \bigoplus M_i\rightarrow \lim\limits_{\longrightarrow}M_i\rightarrow 0$ where $K=\langle x-f_{ij}(x)\mid x\in M_i,i\leq j\in I\rangle$ (see \cite[(2.1.1)]{gt}). Note that  $\{(rM_i,f_{ij})\mid i,j\in\Lambda\}$ is also a direct system of $R$-modules. We have the following equivalence $$\lim\limits_{\longrightarrow}(rM_i)\cong \frac{\bigoplus rM_i}{K'}=\frac{r\bigoplus M_i}{rK}\cong r \frac{\bigoplus M_i}{K}\cong r\lim\limits_{\longrightarrow}M_i$$  where $K'=\langle rx-f_{ij}(rx)\mid rx\in rM_i,i\leq j\in I\rangle$.
\end{proof}

\begin{lemma}\label{inj-cog}
Let $E$ be an injective cogenerator. Then the following assertions are equivalent.
\begin{enumerate}
    \item $T$ is  uniformly $S$-torsion with respect to $s$.
      \item   $\Hom_R(T,E)$  is  uniformly $S$-torsion with respect to $s$.
\end{enumerate}
\end{lemma}
\begin{proof} $(1)\Rightarrow (2)$: Follows from \cite[Lemma 4.2]{QKWCZ21}.

$(2)\Rightarrow (1)$: Let $f:sT\rightarrow E$ be an $R$-homomorphism and $i:sT\rightarrow T$ the embedding map. Since $E$ is injective, there exists an an $R$-homomorphism  $g:T\rightarrow E$ such that $f=gi$. Let $st\in sT$, we have $f(st)=sg(t)=0$ as  $s\Hom_R(T,E)=0$ . So  $\Hom_R(sT,E)=0$. Hence $sT=0$ as $E$ is an injective cogenerator.
\end{proof}

A  multiplicative subset  $S$ of $R$  is said to be \emph{regular}  if it is composed of non-zero-divisors. 
Next, we give some new characterizations of $u$-$S$-flat modules
\begin{proposition}\label{flat-FP-injective}
Let $R$ be a ring, $S$ a multiplicative subset of $R$. Then the following assertions are equivalent.
\begin{enumerate}
    \item $F$ is $u$-$S$-flat.
     \item there exists an element $s\in S$ satisfying that $\Tor_1^R(N,F)$ is uniformly $S$-torsion with respect to $s$ for any finitely presented $R$-module $N$.
      \item   $\Hom_R(F,E)$ is $u$-$S$-injective for any injective module $E$.
     \item   $\Hom_R(F,E)$ is $u$-$S$-absolutely pure for any injective module $E$.
    \item  if $E$ is an injective cogenerator, then $\Hom_R(F,E)$ is $u$-$S$-injective.
       \item  if $E$ is an injective cogenerator, then $\Hom_R(F,E)$ is $u$-$S$-absolutely pure.
\end{enumerate}
Moreover, if $S$ is regular, then all above are equivalent to the following assertions.

\ $(7)$\  there exists  $s\in S$ satisfying that $\Tor_1^R(R/I,F)$ is uniformly $S$-torsion with respect to $s$ for any ideal $I$ of $R$.

\ $(8)$\    there exists  $s\in S$ satisfying that, for any ideal $I$ of $R$, the natural homomorphism  $\sigma_I:I\otimes_RF\rightarrow IF$ is a $u$-$S$-isomorphism with respect to $s$.

\ $(9)$\ there exists  $s\in S$ satisfying that $\Tor_1^R(R/K,F)$ is uniformly $S$-torsion with respect to $s$ for any finite generated ideal $K$ of $R$.

\ $(10)$\    there exists  $s\in S$ satisfying that, for any finite generated ideal $K$ of $R$, the natural homomorphism  $\sigma_K:K\otimes_RF\rightarrow KF$ is a $u$-$S$-isomorphism with respect to $s$.
\end{proposition}
\begin{proof}
$(1)\Rightarrow (2)$: Set the set $\Gamma=\{(K,R^n)\mid K$ is a finitely generated submodule of $ R^n$ and $n<\infty\}.$ Define $M=\bigoplus\limits_{(K,R^n)\in \Gamma}R^n/K$. Then  $s\Tor_1^R(M,F)=s\bigoplus\limits_{(K,R^n)\in \Gamma}\Tor_1^R(R^n/K,F) =0$ for some $s\in S$. Let $N$ be a finitely presented $R$-module, then $N\cong R^n/K$ for some $(K,R^n)\in \Gamma$. Hence $\Tor_1^R(N,F)=0$ is uniformly $S$-torsion with respect to $s$.

$(2)\Rightarrow (1)$:  Let $M$ be an $R$-module. Then $M=\lim\limits_{\longrightarrow}N_i$ for some direct system of finitely presented $R$-modules $\{N_i\}$. So $s\Tor_1^R(M,F)=s\Tor_1^R(\lim\limits_{\longrightarrow}N_i,F)\cong s(\lim\limits_{\longrightarrow}\Tor_1^R(N_i,F))\cong \lim\limits_{\longrightarrow}(s\Tor_1^R(N_i,F))=0$ by Lemma \ref{ele-dire}. Hence $F$ is $u$-$S$-flat by \cite[Theorem 3.2]{zwz21}

$(1)\Rightarrow (3)$: Let $M$ be an $R$-module and $E$ an injective $R$-module. Since $M$ is $u$-$S$-flat,  then $\Tor_1^R(M,F)$  is uniformly $S$-torsion. Thus $\Ext_R^1(M,\Hom_R(F,E))\cong\Hom_R(\Tor_1^R(M,F),E)$ is also uniformly $S$-torsion by \cite[Lemma 4.2]{QKWCZ21}.  Thus $\Hom_R(F,E)$ is $u$-$S$-injective by \cite[Theorem 4.3]{QKWCZ21}.

$(3)\Rightarrow (4)\Rightarrow (6)$ and $(3)\Rightarrow (5)\Rightarrow (6)$: Trivial.

$(6)\Rightarrow (2)$:  Let  $E$ an injective cogenerator. Since $\Hom_R(F,E)$ is $u$-$S$-absolutely pure,  there exists  $s\in S$ such that  $\Hom_R(\Tor_1^R(N,F),E)\cong \Ext_R^1(N,\Hom_R(F,E))$ is  uniformly $S$-torsion with respect to $s$ for any finitely presented $R$-module $N$. Since $E$ is an injective cogenerator, $\Tor_1^R(N,F)$ is uniformly $S$-torsion with respect to $s$ for any finitely presented $R$-module $N$ by Lemma \ref{inj-cog}.

$(2)\Rightarrow (9)$, $(7)\Rightarrow (9)$, $(7)\Leftrightarrow (8)$ and $(9)\Leftrightarrow (10)$: Obvious.

$(10)\Rightarrow (8)$: Let $\sum\limits_{i=1}^na_i\otimes x_i\in \Ker(\sigma_I)$. Let $K$ be the finitely generated ideal generated by $\{a_i\mid i=1,\cdots,n\}$.
Consider the following commutative diagram:

$$\xymatrix@R=25pt@C=25pt{
K\otimes_RF\ar[d]^{\sigma_K}\ar[r]^{i\otimes 1}&I\otimes_RF \ar[d]^{\sigma_I} \\
KF\ar[r]^{i'}&IF \\ }$$
Let $\sum\limits_{i=1}^na_i\otimes x_i$ be the element in $K\otimes_RF$ such that $i\otimes 1(\sum\limits_{i=1}^na_i\otimes x_i)=\sum\limits_{i=1}^na_i\otimes x_i\in I\otimes_RF$.
Since $i' \sigma_K(\sum\limits_{i=1}^na_i\otimes x_i)= \sigma_I (i\otimes 1(\sum\limits_{i=1}^na_i\otimes x_i))=\sigma_I (\sum\limits_{i=1}^na_i\otimes x_i)=0$, we have $\sum\limits_{i=1}^na_i\otimes x_i\in \Ker(\sigma_K)$ since $i' $ is a monomorphism. Then  $s\sum\limits_{i=1}^na_i\otimes x_i=0\in K\otimes_RF$. So $s\sum\limits_{i=1}^na_i\otimes x_i=s i\otimes 1(\sum\limits_{i=1}^na_i\otimes x_i)=i\otimes 1(s\sum\limits_{i=1}^na_i\otimes x_i)=0 \in I\otimes_RF$. Hence $s \Ker(\sigma_I)=0$.

Now assume the multiplicative subset  $S$ is regular.

 $(7)\Rightarrow (5)$ Let $E$ be an injective cogenerator. Since  $\Tor_1^R(R/I,F)$ is uniformly $S$-torsion with respect to $s$, we have $\Hom_R(\Tor_1^R(R/I,F),E)\cong \Ext_R^1(R/I,\Hom_R(F,E))$ is uniformly $S$-torsion with respect to $s$ by Lemma \ref{inj-cog}. Since $s$ is regular and $E$ is injective,  we have $E$ is $s$-divisible. So  $\Hom_R(F,E)$ is also $s$-divisible.  Hence $\Hom_R(F,E)$ is $u$-$S$-injective by \cite[Proposition 4.9]{QKWCZ21}.
\end{proof}

In 1960, Chase also characterized coherent rings in terms of flat modules (see \cite[Theorem 2.1]{C60}). Now, we are ready to give  a ``uniform'' $S$-version of Chase Theorem.

\begin{theorem}\label{w-coh-chase} \textbf{$($Chase Theorem for $u$-$S$-coherent rings$)$}
Let $R$ be a ring and $S$ is a regular  multiplicative subset of $R$. Then the following assertions are equivalent:
\begin{enumerate}
    \item $R$ is a $u$-$S$-coherent ring.
       \item there is $s\in S$ such that any direct product of flat modules is $u$-$S$-flat with respective to $s$.
       \item  there is $s\in S$ such that any direct product of projective modules is $u$-$S$-flat  with respective to $s$.
       \item  there is $s\in S$ such that any direct product of $R$ is $u$-$S$-flat  with respective to  $s$.
\end{enumerate}
\end{theorem}
\begin{proof}

$(2)\Rightarrow (3)\Rightarrow (4)$ Trivial.

$(1)\Rightarrow (2)$ Suppose $R$ is $u$-$S$-coherent with respective to some $s\in S$. Let $\{F_i\mid i\in \Lambda\}$ be a family of flat $R$-modules and $I$ a finitely generated ideal of $R$. Then $I$ is $u$-$S$-finitely presented with respective to $s$. So we have an exact sequence $0\rightarrow T'\rightarrow K\xrightarrow{f} I\rightarrow T\rightarrow 0$ with $K$ finitely presented and $sT=sT'=0$. Set $\Im(f)=K'$. Consider the following commutative diagrams with rows exact:
$$\xymatrix{
& T'\otimes_R \prodi F_i\ar[d]_{\alpha}\ar[r]^{} & K\otimes_R  \prodi F_i \ar[d]_{\gamma}^{\cong}\ar[r]^{} & K'\otimes_R  \prodi F_i \ar[d]^{\beta}\ar[r]^{} &  0\\
   0 \ar[r]^{} & \prodi (T'\otimes_R F_i )\ar[r]^{} &\prodi (K\otimes_R F_i ) \ar[r]^{} & \prodi (K'\otimes_R F_i) \ar[r]^{} &  0,\\}$$
   and
$$\xymatrix{
& K'\otimes_R \prodi F_i\ar[d]_{\beta}\ar[r]^{} & I\otimes_R  \prodi F_i \ar[d]^{\theta}\ar[r]^{} & T\otimes_R  \prodi F_i \ar[d]^{}\ar[r]^{} &  0\\
   0 \ar[r]^{} & \prodi (K'\otimes_R F_i )\ar[r]^{} &\prodi (I\otimes_R F_i ) \ar[r]^{} & \prodi (T\otimes_R F_i) \ar[r]^{} &  0,\\}$$
By \cite[Lemma 3.8(2)]{gt}, $\gamma$ is an isomorphism. Then $\Ker(\beta)\cong \Coker(\alpha)$ which is $u$-$S$-torsion with respective to $s$. Since $K'$ is finitely generated, we have $\beta$ is an epimorphism by \cite[Lemma 3.8(1)]{gt}. Since $T\otimes_R  \prodi F_i$ and $\Ker(\beta)$ are all $u$-$S$-torsion with respective to $s$, so $\Ker(\theta)$ is also  $u$-$S$-torsion  with respective to $s$.

Now we consider the  following commutative diagram with rows exact:
$$\xymatrix{
0\ar[r]& \Tor_1^R(R/I,\prodi F_i)\ar[d]_{}\ar[r]^{} & I\otimes_R \prodi F_i \ar[d]^{\theta}\ar[r]^{} & R\otimes_R  \prodi F_i \ar[d]^{}\\
   0 \ar@{=}[r] & \prodi\Tor_1^R(R/I, F_i)\ar[r]^{} &\prodi (I\otimes_R F_i ) \ar[r]^{} & \prodi (R\otimes_R F_i),\\}$$
Note $\Tor_1^R(R/I,\prodi F_i)\subseteq \Ker(\theta)$. So $\Tor_1^R(R/I,\prodi F_i)$ is $u$-$S$-torsion with respective to $s$, Hence  $\prodi F_i$ is $u$-$S$-flat (with respective to $s$) by Proposition \ref{flat-FP-injective}.

$(4)\Rightarrow (1)$ Let $I$ be a finitely generated ideal of $R$.  Consider the  following commutative diagram with rows exact:
$$\xymatrix{
&I\otimes_R \prodi R\ar[d]_{g}\ar[r]^{f} & R\otimes_R  \prodi R \ar[d]_{}^{\cong}\ar[r]^{} & R/I\otimes_R  \prodi R \ar[d]_{}^{\cong}\ar[r]^{} &  0\\
   0 \ar[r]^{} & \prodi (I\otimes_RR)\ar[r]^{} &\prod_{i\in I} (R\otimes_RR ) \ar[r]^{} & \prodi (R/I\otimes_RR ) \ar[r]^{} &  0.\\}$$
Since $\prodi R$ is a $u$-$S$-flat module with respective to $s$, then $f$ is a $u$-$S$-monomorphism. So $g$ is also a $u$-$S$-monomorphism with respective to $s$.

Let $0\rightarrow L\rightarrow F\rightarrow I\rightarrow 0$ be an exact sequence with $F$ finitely generated free.
 Consider the  following commutative diagram with rows exact:
$$\xymatrix{
&L\otimes_R \prodi R\ar[d]_{h}\ar[r]^{} & F\otimes_R  \prodi R \ar[d]_{}^{\cong}\ar[r]^{} & I\otimes_R  \prodi R \ar[d]_{}^{g}\ar[r]^{} &  0\\
   0 \ar[r]^{} & \prodi (L\otimes_RR)\ar[r]^{} &\prod_{i\in I} (F\otimes_RR ) \ar[r]^{} & \prodi (I\otimes_RR ) \ar[r]^{} &  0.\\}$$
Since $g$ is a  $u$-$S$-monomorphism with respective to $s$, $h$ is a $u$-$S$-epimorphism with respective to $s$. Set $\Lambda$  equal to the cardinal of $L$. We will show $L$ is $S$-finite with respective to $s$. Indeed, consider the following exact sequence
$$\xymatrix{
L\otimes_R  R^L \ar[rr]^{h}\ar@{->>}[rd] &&L^L \ar[r]^{} & T\ar[r]^{} &  0\\
    &\Im h \ar@{^{(}->}[ru] &&  &   \\}$$
with $T$ a $u$-$S$-torsion module with respective to $s$.
Let $x=(m)_{m\in L}\in L^L$. Then  $sx\subseteq \Im h$.
Subsequently, there exist $m_j\in L, r_{j,i}\in R, i\in L, j=1,\cdots,n$
such that for each $t=1,\dots,k$, we have
$$sx=h(\sum_{j=1}^n m_j\otimes (r_{j,i})_{i\in L})=(\sum_{j=1}^n m_j r_{j,i})_{i\in L}.$$
Set $U=\langle m_j\mid j=1,\dots,n\rangle $ be the finitely generated submodule of $L$. Now, for any $m\in L$, $sm\in \langle \sum_{j=1}^n m_j r_{j,m} \rangle\subseteq  U$, thus the embedding map $U\hookrightarrow L$ is a $u$-$S$-isomorphism with respective to $s$ and so $L$ is  $S$-finite  with respective to $s$. Consequently, $I$ is $u$-$S$-finitely presented with respective to $s$. Hence,  $R$ is $u$-$S$-coherent with respective to $s$.
\end{proof}

In 1982, Matlis \cite[Theorem 1]{M82} showed that a ring $R$ is  coherent   if and only if $\Hom_R(M,E)$ is  flat for any injective modules $M$ and $E$. The rest of this paper is devoted to obtain a ``uniform''  $S$-version of this result.

\begin{lemma}\label{qus-usf}
Let $R$ be a ring, $S$ is a regular  multiplicative subset of $R$ and $E$ an injective cogenerator over $R$. Suppose $\Hom_R(E,E)$ is  $u$-$S$-flat with respective to $s\in S$,  then $\Hom_R(E,E)/R$ is also  $u$-$S$-flat with respective to $s$.
\end{lemma}
\begin{proof} Let $I$ be an ideal of $R$. Set $H=\Hom_R(E,E)$. Let $i:R\rightarrowtail H$ be the multiplicative map. Suppose $H$ is  $u$-$S$-flat with respective to $s\in S$. Then there is a long exact sequence  $$\Tor_1^R(R/I,H)\rightarrow\Tor_1^R(R/I,H/R)\rightarrow R/I\otimes_RR\xrightarrow{R/I\otimes i} R/I\otimes H.$$
Note that $\Ker(R/I\otimes i)\cong (HI\cap R)/I=0$ by \cite[Proposition 1(2)]{M82}. Since $\Tor_1^R(R/I,H)$ is  $u$-$S$-torsion with respective to $s\in S$, $\Tor_1^R(R/I,H/R)$ is  $u$-$S$-torsion with respective to $s\in S$, which implies that $H/R$ is also  $u$-$S$-flat with respective to $s$.
\end{proof}

\begin{lemma}\label{dirc}
Let $R$ be a ring, $S$ is a regular  multiplicative subset of $R$. Suppose that $\{A_\lambda\mid \lambda\in\Lambda\}$ is a family of   $u$-$S$-flat modules with respective to $s\in S$, and that $B_\lambda$ is a submodule of $A_\lambda$ such that $A_\lambda/B_\lambda$ is  $u$-$S$-flat with respective to $s$ for each $ \lambda\in\Lambda$. Then   $\prod\limits_{\lambda\in\Lambda}A_\lambda$ is $u$-$S$-flat with respective to $s$ if and only if so is $\prod\limits_{\lambda\in\Lambda}B_\lambda$ and $\prod\limits_{\lambda\in\Lambda}A_\lambda/B_\lambda$.
\end{lemma}
\begin{proof} Let   $I$ be a  finitely generated ideal of $R$. Then there is an exact sequence $$\Tor_2^R(R/I,\prod\limits_{\lambda\in\Lambda}A_\lambda/B_\lambda) \rightarrow \Tor_1^R(R/I,\prod\limits_{\lambda\in\Lambda}B_\lambda) \rightarrow \Tor_1^R(R/I,\prod\limits_{\lambda\in\Lambda}A_\lambda).$$
By \cite[Theorem 3.2]{zwz21}, we just need to show $\prod\limits_{\lambda\in\Lambda}A_\lambda/B_\lambda$ is $u$-$S$-flat with respective to $s$. Consider the following exact sequence  $$\Tor_1^R(R/I,\prod\limits_{\lambda\in\Lambda}A_\lambda)\rightarrow \Tor_1^R(R/I,\prod\limits_{\lambda\in\Lambda}A_\lambda/B_\lambda)\rightarrow R/I\otimes_R \prod\limits_{\lambda\in\Lambda}B_\lambda\xrightarrow{f} R/I\otimes_R \prod\limits_{\lambda\in\Lambda}A_\lambda.$$

Since $\Tor_1^R(R/I,\prod\limits_{\lambda\in\Lambda}A_\lambda)$ is $u$-$S$-torsion with respective to $s$, to show $\prod\limits_{\lambda\in\Lambda}B_\lambda$ is  $u$-$S$-flat with respective to $s$, we just need to show $\Ker(f)$ is $u$-$S$-torsion with respective to $s$.
Note that
$\Ker(f)\cong (\prod_{\lambda\in\Lambda}B_\lambda\cap I(\prod_{\lambda\in\Lambda}A_\lambda))/I \prod_{\lambda\in\Lambda}B_\lambda
\cong \prod_{\lambda\in\Lambda} (B_\lambda\cap IA_\lambda)/I B_\lambda$
as $I$ is finitely generated.  Consider the following exact sequence  $\Tor_1^R(R/I,A_\lambda)\rightarrow \Tor_1^R(R/I,A_\lambda/B_\lambda)\rightarrow R/I\otimes_R B_\lambda\xrightarrow{f_\lambda} R/I\otimes_R \prod A_\lambda.$ We have $\Ker(f_\lambda)\cong (B_\lambda\cap IA_\lambda)/I B_\lambda$ is $u$-$S$-torsion with respective to $s$ since $A_\lambda/B_\lambda$ is  $u$-$S$-flat with respective to $s$ . So $\Ker(f)\cong \prod_{\lambda\in\Lambda} \Ker(f_\lambda)$ is $u$-$S$-torsion with respective to $s$.
 \end{proof}

\begin{theorem}\label{phi-coh-fp} \textbf{$($Matlis Theorem for $u$-$S$-coherent rings$)$}
Let $R$ be a ring and $S$ is a regular  multiplicative subset of $R$.  Then the following statements are equivalent.
\begin{enumerate}
    \item $R$ is a $u$-$S$-coherent ring.
    \item  there are $s_1,s_2\in S$ such that  $\Hom_R(M,E)$ is $u$-$S$-flat with respective to $s_1$ for any  $M$ a  $u$-$S$-absolutely pure module  with respective to $s_2$  and any injective module $E$.
   \item   there are $s_1,s_2\in S$ such that $\Hom_R(M,E)$ is $u$-$S$-flat   with respective to $s_1$ for any $M$  $u$-$S$-injective module with respective to $s_2$    and any injective module $E$.
   \item   there is $s_1,s_2\in S$ such that if $E$  is injective cogenerators, then   $\Hom_R(M,E)$ is $u$-$S$-flat  with respective to $s_1$ for any $M$  $u$-$S$-injective module with respective to $s_2$.
       \item  there are $s_1,s_2\in S$ such that  $\Hom_R(\Hom_R(M,E_1),E_2)$ is $u$-$S$-flat  with respective to $s_1$  for any $M$ a $u$-$S$-flat module with respective to $s_2$     and any injective modules $E_1,\ E_2$.
         \item  there are $s_1,s_2\in S$ such that   if $E_1$ and $E_2$ are injective cogenerators, then $\Hom_R(\Hom_R(M,E_1),E_2)$ is $u$-$S$-flat with respective to $s_1$   for any  $M$ a $u$-$S$-flat module with respective to $s_2$.
\item   there is $s\in S$ such that if $E_1$is an injective cogenerator, then $\Hom_R(E_1,E_2)$ is $u$-$S$-flat  with respective to $s$ for  any injective cogenerator$E_2$.
\end{enumerate}
\end{theorem}
\begin{proof}
$(2)\Rightarrow (3)\Rightarrow (4)\Rightarrow (7)$ and $(5)\Rightarrow (6)$: Trivial.

$(3)\Leftrightarrow (5)$ and $(4)\Leftrightarrow (6)$: Follows from Proposition \ref{flat-FP-injective}.

 $(1)\Rightarrow (2)$: Suppose $R$ is a uniformly $S$-coherent ring with respective to some element $s\in S$. Let   $I$ be a  finitely generated ideal of $R$. Then we have an exact sequence $0\rightarrow T'\rightarrow K\xrightarrow{f} I\rightarrow T\rightarrow 0$ with $K$ finitely presented and $sT=sT'=0$. Set $\Im(f)=K'$.   Consider the following commutative diagrams with exact rows ($(-,-)$ is instead of $\Hom_R(-,-)$):
 $$\xymatrix@R=20pt@C=15pt{
(M,E)\otimes_R T'\ar[r]^{} \ar[d]^{\psi^1_{T'}} &(M,E)\otimes_R K \ar[d]_{\psi_{K}}^{\cong}\ar[r]^{} & (M,E)\otimes_R K' \ar[d]_{\psi_{K'}}\ar[r]^{}&0\\
((T',M),E)\ar[r]^{} & ((K,M),E) \ar[r]^{} &((K',M),E)\ar[r]^{} &0,\\ }$$

$$\xymatrix@R=20pt@C=15pt{
0\ar[r]^{}&\Tor_1^R((M,E),R/K') \ar[r]^{}\ar[d]^{\psi^1_{R/K'}} &(M,E)\otimes_R K' \ar[d]_{\psi_{K'}}\ar[r]^{} & (M,E)\otimes_R R \ar[d]_{\psi_{R}}^{\cong}\ar[r]^{} &(M,E)\otimes_R R/K' \ar[d]_{\psi_{R/K'}}^{}\ar[r]^{} &0\\
0\ar[r]&(\Ext_R^1(R/K',M),E) \ar[r]^{} &((K',M),E)\ar[r]^{} & ((R,M),E) \ar[r]^{} &((R/K',M),E)\ar[r]^{} &0\\ }$$
and
$$\xymatrix@R=20pt@C=15pt{
\Tor_1^R((M,E),T)\ar[d]\ar[r]^{}&\Tor_1^R((M,E),R/K') \ar[r]^{}\ar[d]^{\psi^1_{R/K'}} &\Tor_1^R((M,E),R/I) \ar[d]_{\psi^1_{R/I}}\ar[r]^{} & (M,E)\otimes_R T \ar[d]\\
(\Ext_R^1(T,M),E)\ar[r]&(\Ext_R^1(R/K',M),E) \ar[r]^{} &(\Ext_R^1(R/I,M),E)\ar[r]^{} & ((T,M),E) \\ }$$ Since $\psi_{K}$ is an isomorphism by  \cite[Proposition 8.14(1)]{hh} and \cite[Theorem 2]{ELMUT1969},  $\psi_{K'}$ is a  $u$-$S$-isomorphism with respective to $s$, and so is $\psi^1_{R/K'}$. Then  $\psi^1_{R/I}$  is a  $u$-$S$-isomorphism with respective to $s^3$ (see the proof of \cite[Theorem 1.2]{z21-swd}). Since $M$ is $u$-$S$-absolutely pure, $\Ext_R^1(R/I,M)$ is $u$-$S$-torsion with respective to  $s_2$ ($s_2$ is independent of $I$). Then $\Tor_1^R(\Hom_R(M,E),R/I)$ is $u$-$S$-torsion with respective to  $s_1:=s^3s'$, and thus $\Hom_R(M,E)$ is $u$-$S$-flat  with respective to  $s_1$ by Proposition \ref{flat-FP-injective}.

$(7)\Rightarrow (1)$: Let $E$ be an injective cogenerator and set
$H = \Hom_R(E,E)$. Then $H$ is $u$-$S$-flat with respective to $s$ by assumption.  Since $R\subseteq H$, we have that $H/R$ is $u$-$S$-flat with respective to $s$ by Lemma \ref{qus-usf}. Let $\Lambda$ be an index set. Set $ H_{\lambda} = H$, $R_{\lambda}= R$ and $E_{\lambda} = E$ for any $\lambda\in \Lambda$. Since $\prod\limits_{\lambda\in \Lambda}E_{\lambda} $ is also a injective cogenerator, $\prod\limits_{\lambda\in \Lambda} H_{\lambda}\cong \Hom_R(E_{\lambda},\prod\limits_{\lambda\in \Lambda}E_{\lambda} )$ is $u$-$S$-flat with respective to $s$ by assumption. Hence $\prod\limits_{\lambda\in \Lambda}R_{\lambda}$ is $u$-$S$-flat with respective to $s$ by Lemma \ref{dirc}. So $R$ is a $u$-$S$-coherent ring by Theorem \ref{w-coh-chase}.
\end{proof}


\end{document}